\documentclass[a4paper,12pt]{article}
\usepackage{amsmath}
\usepackage{amsthm}
\usepackage{amssymb}  \usepackage{latexsym}
\usepackage{eufrak}
\usepackage{euscript}
\usepackage{graphicx}
\usepackage[english]{babel}
\usepackage[utf8]{inputenc}
\usepackage[T1]{fontenc}
\usepackage{enumerate}
\usepackage{color}
\usepackage{setspace}
\usepackage{authblk}
\newtheorem{theorem}{Theorem}[subsection]
\newtheorem{lemma}{Lemma}
\newtheorem{proposition}{Proposition}

\newtheorem{remark}{Remark}
\newtheorem{corollary}{Corollary}[subsection]
\newtheorem{exm}{Example}[subsection]

\newtheorem{Question}{Question}
\usepackage[a4paper,top=3cm,bottom=2cm,left=1.5cm,right=1.5cm,marginparwidth=1cm]{geometry}

\def\H{{\mathcal H}}
\title{Joins of Hypergraphs and Their Spectra}
\author{\rm Amitesh Sarkar, Anirban Banerjee}
\affil{Department of Mathematics and Statistics
	\\Indian Institute of Science Education and Research Kolkata\\ Mohanpur-741246,  India
	\\ {\footnotesize \textit { 16rs059@iiserkol.ac.in, anirban.banerjee@iiserkol.ac.in}}
}

\begin{document}

\maketitle

\begin{abstract}
Here, we represent a general hypergraph by a matrix and study its spectrum. We extend the definition of equitable partition and joining operation for hypergraphs, and use those to compute eigenvalues of different hypergraphs. We derive the characteristics polynomial of a complete $m$-uniform $m$-partite hypergraph $K^m_{n_1,n_2,\dots,n_m}$.  Studying edge corona of hypergraphs we find the complete spectrum of $s$-loose cycles $C^m_{L(s;n)}$ for $m \geq 2s+1$
and the characteristics polynomial of a $s$-loose paths  $P^{(m)}_{L(s;n)}$. Some of the eigenvalues of $P^{(m)}_{L(s;n)}$ are also derived. Moreover, using vertex corona, we show how to generate infinitely many pairs of non-isomorphic co-spectral hypergraphs.
\end{abstract}

\section{Introduction}
Spectral graph theory is a well-known devolved area of research in mathematics where one of the interests is to investigate the relationship between the spectrum of different matrices representing graphs with their structure. We also propose various graph operations, for example, graph joining, etc., 
and find the eigenvalues of newly generated graphs by those operations. Here we explore the same for hypergraphs.

As graphs, hypergraphs also have a vertex set and an edge set, but, unlike a graph, an edge of a hypergraph is a nonempty subset of its vertex set.
Hypergraphs can be represented by different hypermatrices or tensors \cite{Cooper2012, Qi2017}  as well by various matrices \cite{Agarwal, Rodriguez}. 
Here we consider a matrix representation of hypergraphs suggested in \cite{Banerjee2017} and find the spectrum of hypergraphs formed by joining operations. In graph theory, 
graph joining  \cite{Cardosoa} plays an important role in constructing certain graphs. Here we define a general hypergraph joining operation and extend it in different scenarios, such as a weighted or un-weighted joining of a set of weighted or un-weighted hypergraphs, the building of a hypergraph of hypergraphs, etc.
We reframe the concept of equitable partition defined on graphs \cite{Simic, Godsil1997, Kuduse}, for hypergraphs to investigate the spectrum of different joins of hypergraphs.
A hypergraph consists of only the edges with the same cardinality $m$, is called $m$-uniform hypergraph. An $m$-uniform hypergraph (on $n$ vertices) having all possible edges is called a complete $m$-uniform hypergraph and is denoted by $K^m_n$.  
Furthermore  $\overline{K}^m_n$ is called the complement of $K^m_n$ 
 which is also a complete $m$-uniform hypergraph with $n$ vertices such that $e$ is an edge of $\overline{K}^m_n$ if and only if it is not an edge of $K^m_n$. 
A graph is known as complete multi($k$)-partite if the vertex set is partitioned into more than one ($k$) parts and the edge set consists of only the all possible edges with one terminal vertex is in one partite and another terminal vertex is in a different partite. Complete multipartite graphs contain a class of integral graphs \cite{Rotman1984, Delorme2012, Hic2008}. It is still challenging to find all the eigenvalues of a complete multipartite graph, although its characteristic polynomial is known \cite{Delorme2012}.
Now, a complete $m$-uniform $k$-partite hypergraph, $K^m_{n_1,n_2,\dots,n_k}$, with the vertex set $V$ partitioned into $k$ parts say  $\{V_1,V_2,\dots,V_k\}$, is called  complete $m$-uniform  weak(strong) $k$-partite hypergraph if $m\geq k$($m\le k$) and the edge set $E=\{e\subseteq V: |e|=m, |e\cap V_i|\geq(\le) 1, \text{for all}~ i=1,2,\dots,k.\}$. Here, we are interested to compute the characteristic polynomial of a 
complete $m$-uniform $m$-partite hypergraph $K^{m}_{n_1,n_2,\dots,n_m}$.
The concept of corona of two graphs was introduced by R. Fruchut and F. Harary \cite{Harary1970}. Here we also defined the generalized corona of hypergraphs.
Vertex corona of two graphs can provide infinitely many pairs of non-isomorphic co-spectral graphs \cite{McLeman2011, Barik2007, Barik2017, Laali2016}. We derive the same for two hypergraphs. 
Edge corona, which is an interesting graph operation \cite{Hou2010, Lou}, is defined for hypergraphs and is applied to find the eigenvalues of $s$-loose paths and $s$-loose cycles. 
%%%%%%%%%%%%
Generalized $s$-loose path $P^{(m)}_{L(s;n)}$ is an  $m$-uniform hypergraph with the vertex set $V=\{1,2,\dots,m+(n-1)(m-1)\}$ 
and edge set the  $E=\{\{i(m-s)+1,i(m-s)+2,\dots,i(m-s)+m\}:i=0,1,\dots,n-1\}$  \cite{Peng2016}. For $s=1$, $P^{(m)}_{L(1;n)}$ is known as loose path.
Similarly, generalized $s$-loose cycle $C^m_{L(s;n)}$ is an $m$-uniform hypergraph with the vertex set $V=\{1,2,\dots,n(m-s)\}$ and the edge set 
$\{\{i(m-s)+1,\dots,i(m-s)+m\}:i=0,1,\dots,n-1\}\cup{\{n(m-s)-s+1,n(m-s)-s+2,\dots,n(m-s),1,2,\dots,s\}}$. 
In  spectral study of hypergraphs, 
using hypermatrix, it is not easy to find the eigenvalues of $P^{(m)}_{L(s;n)},C^{(m)}_{L(s;n)}$, hyper-cubes. In \cite{Luke2019}, the authors considered oriented hypergraphs and proved some results on the spectrum of adjacency matrices of $P^{(m)}_{L(s;n)},C^{(m)}_{L(s;n)}$. 
They also posed some open questions regarding the spectrum of $P^{(m)}_{L(s;n)},C^{(m)}_{L(s;n)}$. 
Here we  find the spectrum of $C^{(m)}_{L(s;n)}$ for $m\geq{2s+1}$. We also compute some eigenvalues of an $s$-loose path and the remaining eigenvalues are shown as the zeros of a polynomial.\\

In this article, the Perron-Frobenius theorem for nonnegative irreducible matrices has been extensively used.
\begin{lemma}[Perron-Frobenius Theorem, Chapter 8, \cite{Horn} ]\label{th0} Let $A$ be a non-negative irreducible matrix. Then
  \begin{itemize}
   \item $A$ has a positive eigenvalue $\lambda$ with positive eigenvector.
   \item{$\lambda$ is simple} and for any other eigenvalue $\mu$ of $A$,~$|\mu|\leq{\lambda}.$
  \end{itemize}
\end{lemma}
This positive eigenvalue known as Perron eigenvalue and others as non-Perron eigenvalue.

\section{Adjacency Matrix for Hypergraphs}
Let $\H(V,E,W)$ be a hypergraph with the vertex set $V=\{1,2,\dots ,n\},$ the edge set $E$ and with an weight function
$W:E\rightarrow{\mathbb{R}_{\geq{0}}}$  defined by $W(e)=w_e$ for all $e\in{E}.$ The adjacency matrix $A_{\H}$ of $\H$ is defined as 
$$(A_{\H})_{ij}:=\sum\limits_{e\in{E};i,j\in{e}}\dfrac{w_e}{|e|-1}.$$
This definition is adopted from the definition of the adjacency matrix of an unweighted non-uniform hypergraph defined in \cite{Banerjee2017}.
The valency of a vertex $i$ of $H$ is given by $d_i:=\sum\limits_{e\in{E},i\in{e}}w_e$ $ = \sum\limits_{j(\neq{i})\in{V}}\sum\limits_{e\in{E};i,j\in{e}}\dfrac{w_e}{|e|-1}$.
Thus for an $m$-uniform hypergraph $\H$ we have  $(A_{\H})_{ij}=\dfrac{1}{m-1}\sum\limits_{e\in{E},i,j\in{e}}w_e$. If we take $w_e=1$, 
then $(A_{\H})_{ij}=\dfrac{d_{ij}}{m-1}$ where $d_{ij}$ is the codegree of the vertices $i,j$, i.e., the number of edges containing both the vertices $i$ and $j$\footnote{In 2002 Rodriguez defined the adjacency matrix and corresponding Laplacian matrix for uniform hypergraphs and which is the matrix $(m-1)A_{\H}$ \cite{Rodriguez}.}. If the vertices $i$ and $j$ belong to an edge we call, they are adjacent and denote as $i\sim j$, otherwise, $i\nsim j$.
For unweighted hypergraph we take $W(e)=1$ for all $e\in{E}$ and then the hypergraph $\H$ is denoted by $\H(V,E)$. For an unweighted hypergraph $\H$ the valency $d_i$ of a vertex $i$ is the degree of $i$, i.e., the number of edges containing the vertex $i$. 
A hypergraph is ($r$-)regular if all the degrees of vertices are equal( to $r$). 
A hypergraph  is connected if for any two vertices $v_0, v_l$ there is an alternating sequence $v_0e_1v_1e_2v_2\dots v_{l-1}e_lv_l$ of distinct vertices $v_0,v_1,v_2,\dots, v_l$ and distinct edges $e_1,e_2,\dots, e_l$, such that, $v_{i-1},v_i \in e_i$ for $i=1,\dots,l$.
For an $r$-regular connected hyperghraph $\H$, with $n$ vertices $(r, 1^t)$ is always an eigenpair of $A_{\H}$ and by Lemma \ref{th0}, $r$ is a simple eigenvalue of $A_{\H}$. Moreover, since $A_{\H}$ is real symmetric matrix, thus we can have an orthogonal set of eigenvectors for $A_{\H}$.

\subsection{Equitable Partition for Hypergraphs}

Let $\mathcal{H}(V,E,W)$ be an $m$-uniform hypergraph. We say a partition $\pi=\{C_1,C_2,\dots,C_k\}$ of $V$ is an equitable partition of $V$
if for any $p,q\in{1,2, \dots ,k}$ and for any $i\in{C_p}$,~$$\sum\limits_{j,j\in{C_q}}(A_\mathcal{H})_{ij}=b_{pq},$$ where $b_{pq}$ is a constant depends
only on $p$ and $q$. Note that any two vertices belonging to the same cell of the partition have the same valency. For an equitable partition with 
$k$-number of cells we define the quotient matrix $B$ as $(B)_{pq}=b_{pq}$, for $1\leq{p,q}\leq{k}$.

\begin{proposition}
Let $\H(V,E,W)$ be an $m$-uniform hypergraph. Then orbits of any subgroup of $Aut(H)$ form an equitable partition.
\end{proposition}

\begin{proof}
Let $\Gamma$ be a subgroup of $Aut{\left(\mathcal{H}\right)}$ 
acts on the vertex set $V=\{1,2,\dots,n\}$. We consider the orbits as $\pi=\{V_1,V_2,\dots,V_k\}$.
For any $i\in{V_p}$ there exists $\varphi\in{\Gamma}$ such that $\varphi\left(i\right)=j\in V_p$. Now we have
\begin{align*}
\sum_{s, s\in{V_q}}(A_\mathcal{H})_{is}
&= \sum_{s,s\in{V_q},s\nsim{i}}(A_\mathcal{H})_{is} +
\sum_{s, s\in{V_q},s\sim{i}}(A_\mathcal{H})_{is}\\
&=\sum_{\substack{\varphi{\left(s\right)},\varphi{\left(s\right)}\in{V_q},\\\varphi{\left(s\right)}
\nsim\varphi{\left(i\right)}}}
(A_\mathcal{H})_{\varphi\left(i\right)\varphi{\left(s\right)}} +
\sum_{\substack{\varphi{\left(s\right)}, \varphi{\left(s\right)}\in{V_q},\\\varphi{\left(s\right)}
	\sim\varphi{\left(i\right)}}}(A_\mathcal{H})_{\varphi\left(i\right)\varphi{\left(s\right)}}\\
&=\sum_{\substack{\varphi{\left(s\right)},\varphi{\left(s\right)}\in{V_q},\\\varphi{\left(s\right)}\nsim{j}}}(A_\mathcal{H})_{j\varphi(s)}+
\sum_{\substack{\varphi{\left(s\right)},\varphi{\left(s\right)}\in{V_q},\\\varphi{\left(s\right)}\sim{j}}}(A_\mathcal{H})_{j\varphi(s)}\\
&=\sum_{t,t\in{V_q},t\nsim{j}}(A_\mathcal{H})_{jt}+
\sum_{t,t\in{V_q},t\sim{j}}(A_\mathcal{H})_{jt}\\
&=\sum_{t,t\in{V_q}}(A_\mathcal{H})_{jt} =constant.
\end{align*}
This completes the proof.
\end{proof}

Let $\mathcal{H}(V,E)$ be an $m$-uniform  weighted hypergraph with $n$~vertices. For an equitable partition $\pi$ with $k$ cells $\{C_1,C_2, \dots ,C_k\}$ 
we define the characteristic matrix $P$ of order $n\times{k}$ as follows
$$
\left(P\right)_{ij}=
\begin{cases}
1  &\text{if vertex i$\in{C_j}$},\\
0  &\text{otherwise.}
\end{cases}
$$
Thus $A_{\mathcal{H}}P=PB$. The column space of $P$ is $A_\mathcal{H}$-invariant. It is easy to show that $B$ 
is also diagonalizable as $A_{\H}$ and for each $\lambda\in{spec(B)}$ with the multiplicity $r$,~$\lambda\in{spec(A_{\H})}$ with the multiplicity atlest $r$.

\section{Weighted Joining of Hypergraphs}
\subsection{Weighted Joining of a Set of Uniform Hypergraphs}
First we consider the joining of two hypergraphs. Let $\H_1(V_1,E_1,W_1)$ and $\H_2(V_2,E_2,W_2)$ be two $m$-uniform hypergraphs. The join 
$\H(V,E,W):=\H_1\oplus^{w_{12}} \H_2$ of $\H_1$ and $\H_2$ is the hypergraph with the vertex set $V=V_1\cup V_2$, edge set 
$E=\cup_{i=0}^2E_i$, where $E_{0}=\{e\subseteq{V} :|e|=m, e\cap{V_i}\not =\phi,\forall i=1,2\}$ and the weight function $W:E\rightarrow \mathbb{R}_{\geq 0}$
is defined by 
$$
W(e)=
\begin{cases}
W_1(e) &\text{if $e\in{E_1}$},\\
W_2(e) &\text{if $e\in{E_2}$},\\
w_{12} &\text{otherwise},
\end{cases}
$$ 
where $w_{12}$ is a real non-negative constant.\\
Now, we define the same operation for a set of $m$-unoform hypergraphs. Let $S=\{\mathcal{H}_i(V_i,E_i,W_i):1\leq{i}\leq{k}\}$,~
$|V_i|=n_i$, be a set of $m$-uniform hypergraphs, $k\leq{m}$. Using the set $S$, of hypergraphs we construct a new $m$-uniform hypergraph $\H(V,E,W)$
where $V=\cup_{i=1}^mV_i$,~$E=\cup_{i=0}^mE_i$,~ $E_{0}=\{e\subseteq{V} : e\cap{V_i}\not =\phi,\forall i={1,2, \dots ,k},|e|=m\}$ and 
$W:E\rightarrow \mathbb{R}_{\geq 0}$ defined by 
$$
W(e)=
\begin{cases}
W_i(e) &\text{if $e\in{E_i}$ for $i=1,2\dots,k$},\\
w_s &\text{otherwise,}
\end{cases}
$$ where $w_s$ is a real non-negative constant. The resultant hypergraph $\H$ is called the join of a set $S$ of $m$-uniform hypergraphs $\H_i$'s and it is denoted as $\mathcal{H}:=\oplus_{\mathcal{H}_i\in{S}}^{w_s}\mathcal{H}_i$.\\
Now we write the adjacency matrix of the join, $\H$. Let $p,q\in{\{1,2,\dots,k\}}$ and $l_1,l_2\in{V_p}$. 
Let
\begin{align*}
d^{S(m)}_{pp}:&=\dfrac{1}{m-1}\text{(number of new edges containing two fixed vertices $l_1,l_2$ from $V_p$)}\\
&=\dfrac{1}{m-1}\sum\limits_{\substack{i_p\geq 0, i_j\geq{1},(j\neq p)\\{i_1+i_2+\dots+i_k=m-2}}}\binom{n_1}{i_1}\binom{n_2}{i_2}\dots\binom{n_{p-1}}{i_{p-1}}\binom{n_p-2}{i_p}\binom{n_{p+1}}{i_{p+1}}\dots\binom{n_k}{i_k}.
\end{align*}
For $p\neq{q}$,~$i\in{Y}={\{1,2,\dots,k\}\backslash{\{p,q\}}}$, we take two fixed vertices $l_1,l_2$ such that $l_1\in{V_p}$,~$l_2\in{V_q}$ and define

\begin{align*}
d^{S(m)}_{pq}&:=\dfrac{1}{m-1}\text{(number of new edges containing two fixed vertices one 
from $V_p$ and another from $V_q$)}\\
&=\dfrac{1}{m-1}\sum\limits_{\substack{i_p,i_q\geq 0, i_j\geq{1},(j\neq p,q)\\{i_1+i_2+\dots+i_k=m-2}}}\binom{n_1}{i_1}\binom{n_2}{i_2}\dots\binom{n_{p-1}}{i_{p-1}}
\binom{n_p-1}{i_p}\binom{n_{p+1}}{i_{p+1}}\dots\\
&\ \ \ \ \ \ \ 
\binom{n_{q-1}}{i_{q-1}}\binom{n_q-1}{i_q}\binom{n_{q+1}}{i_{q+1}}\dots\binom{n_k}{i_k}.
\end{align*}
Therefore the adjacency matirx $A_{\H}$ of $\H$ can be expressed as 
$$
(A_{\mathcal{H}})_{ij}=
\begin{cases}
\left(A_{\mathcal{H}_p}\right)_{ij}+w_sd^{S(m)}_{pp} &\text{if $i, j\in{V_p}$,~$i\neq{j}$},\\
0 &\text{$i=j$},\\
w_sd^{S(m)}_{pq} &\text{if $i\in{V_p}, j\in{V_q}, p\neq{q}.$}
\end{cases}
$$

\subsubsection{Weighted Joining of Uniform Hypergraphs on a Backbone Hypergraph}

Let $\H(V,E,W)$ be an $m$-uniform hyperghraph. We call the hypergraph $\mathcal{H}_{b}(V_b,E_b,W_b)$,~$V_b=\{1,2,\dots,n\}$ as a backbone of $\H$
if $\H$ can be constructed by a set $S=\{\mathcal{H}_i(V_i,E_i,W_i)$: $i=1,2,...,n\}$, of $m$-uniform hypergraphs,~$(m\geq{max\{|e|:e\in{E_b}\}})$ 
with the following operations :
\begin{enumerate}
\item Replace vertex $i$ of $\H_b$ by $\mathcal{H}_i$, for $i=1,2,\dots,n$. 
\item For each edge $e=\{j_1,j_2,\dots,j_k\}\in{E_b}$, take $S_e=\{\mathcal{H}_{j_i}:i=1,2,\dots,k\}$ and apply the operation 
$\oplus^{W_b(e)}_{S_e}$ defined above.
\end{enumerate}
We call the hypergraphs $\mathcal{H}_i$'s as participants on the backbone $\H_b$ to form the hypergraph $\mathcal{H}$. 
Thus the adjacency matrix for $\mathcal{H}$ can be written as
$$
(A_{\mathcal{H}})_{ij}=
\begin{cases}
(A_{\mathcal{H}_p})_{ij}+\sum\limits_{e\in{E_b},p\in{e}}W_b(e)d^{S_e(m)}_{pp} &\text{if $i,j\in{V_p}$,~$i\neq{j}$,}\\
0 &\text{if $i=j$,}\\
\sum\limits_{e\in{E_b}, p, q\in{e}}W_b(e)d^{S_e(m)}_{pq} &\text{if $i\in{V_p}, j\in{V_q}, p\neq{q}$.}
\end{cases}
$$

\begin{theorem}\label{th1}
Let $\mathcal{H}_b(V_b,E_b,W_b)$ be a hypergraph with the vertex set $V_b=\{1,2,\dots,n\}$ and let $\{\mathcal{H}_i(V_i,E_i,W_i):i=1,2,\dots,n\}$ be a collection of regular $m$-uniform
hypergraphs $(m\geq{\{|e|:e\in{E_b}\}})$. Let $\mathcal{H}(V,E,W)$ be the $m$-uniform hypergraph constructed by taking $\H_b$ as backbone hypergraph and 
$\H_i$'s as participants. Then for any non-Perron eigenvalue $\lambda$ of $A_{\mathcal{H}_p}$ with multiplicity $l$,~
$\lambda-\sum\limits_{e\in{E_b},p\in{e}}W_b(e)d^{S_e{(m)}}_{pp}$ is an eigenvalue of $A_{\mathcal{H}}$ with the multiplicity atleast $l$.
\end{theorem}

\begin{proof}
Let $(\lambda,f)$ be an eigenpair of $A_{\H_p}$, such that $f$ is orthiogonal to the constant vector $\left[1,1,1,\dots,1\right]^t$. 
We define $f^*:V\rightarrow{\mathbb{R}}$ by 
$$
f^{*}(v)=
\begin{cases}
 f(v) &\text{if $v\in{V_p}$},\\
 0 &\text{otherwise.}
\end{cases}
$$
Thus $f^*$ is an eigenvector of $A_{\mathcal{H}}$ corresponding to the eigenvalue $\lambda-\sum\limits_{e\in{E_B},p\in{e}}w_ed^{S_e(m)}_{pp}$. 
Since $\sum_{i\in{V_p}}f(i)=0,$ thus the proof folows.
\end{proof}

When $\mathcal{H}_i(V_i,E_i,W_i)$'s are regular, the partition $\pi=\{V_1,V_2,\dots,V_n\}$ forms an equitable partition for
$\mathcal{H}$. In particular if $\mathcal{H}_i$'s are $r_i$ regular then the quotient matrix $B$ is as follows

\begin{align}\label{Qmatrix}
(B)_{pq}=
\begin{cases}
r_p+(n_p-1)\sum\limits_{e\in{E_b},p\in{e}}W_b(e)d^{S_e{(m)}}_{pp} &\text{if $p=q$},\\
n_q\sum\limits_{e\in{E_b},p,q\in{e}}w_b(e)d^{S_e(m)}_{pq} &\text{otherwise.}
\end{cases}
\end{align}
The remaining $n$ eigenvalues of $A_{\mathcal{H}}$ can be obtained from $B$.
\begin{exm}
Take $\H_{b}=K^m_{k}$,~$k\geq{m}$ and $\H_i=\overline{K}^m_{n_i}(V_i,E_i)$,~$i=1,2,\dots,k$. 
Consider $\H_b$ as the backbone and $\H_i$'s as the participants. Then we get the complete $m$-uniform $k$-partite hypergraph $K^m_{n_1,n_2,\dots,n_k}$. By the above theorem we have $0$ is an eigenvalue 
of $A_{K^m_{n_1,n_2,\dots,n_k}}$ with the multiplicity atleast $\sum\limits_{i=1}^kn_i-k$. The quotient matrix $(B)_{k\times k}$ is given by
$$
(B)_{p,q}=
\begin{cases}
0 &\text{if $p=q$,}\\
n_qs_{pq}&\text{otherwise,}
\end{cases}
$$
where $s_{pq}=\sum\limits_{\substack{\{i_1,i_2,\dots,i_l\}\subseteq{\{1,2,\dots,k\}\backslash{\{p,q\}}},\\|\{i_1,i_2,\dots,i_l\}|=m-2}}n_{i_1}n_{i_2}\dots n_{i_l}$
for $p,q=1,2,\dots,k$ and $p\neq q$.
\end{exm}

\begin{corollary}\label{cor1}
Let $S=\{\mathcal{H}_i(V_i,E_i,W_i):1\leq{i}\leq{k}\leq{m}\}$ be a set of $m$-uniform hypergraphs where $\mathcal{H}_i$ 
are $r_i$-regular. Let $\mathcal{H}=\oplus_{\mathcal{H}_i\in{S}}^{w_s}\mathcal{H}_i$. Then for any non-perron eigenvalue $\lambda$ of
$A_{\mathcal{H}_i}$ with the multiplicity $l$,~$\lambda-w_sd^{S(m)}_{ii}$ is an eigenvalue of $A_{\mathcal{H}}$
with multiplicity atleast $l$.
\end{corollary}

Note that the remaining $k$ eigenvalues can be obtained from the quotient martrix $B$ defined as  
$$
(B)_{pq}=
\begin{cases}
r_p+(n_p-1)w_sd^{S(m)}_{pp} &\text{if $p=q$},\\
n_qw_sd^{S(m)}_{pq} &\text{otherwise.}
\end{cases}
$$

\begin{corollary}\label{cor2}
Let $G_b(V_0,E_0)$ be an undirected weighted graph with the adjacency matrix $(w_{ij})_{k\times{k}}$ and $\mathcal{H}_i(V_i,E_i)$'s be  
$r_i$-regular $m$-uniform hypergraphs, for $1\leq{i}\leq{k=|V_0|}$. Let $\mathcal{H}(V,E)$ be the $m$-uniform hypergraph constructed by taking $G_b$ as the backbone
graph and $\H_i$'s as the participants. Then for any non-perron eigenvalue $\lambda$ of $A_{\H_p}$ with multiplicity $l$,~ 
$\lambda-\sum\limits_{e\in{E_0},p\in{e}}w_{pq}d^{S_e(m)}_{pp}$ is an eigenvalue of $A_\mathcal{H}$ with the multiplicity atleast $l$.
\end{corollary}

If we consider $G_b=K_2$ as the backbone and $\mathcal{H}_1,\mathcal{H}_2$ as the participants to construct the hypergraph $\H$ then $\H$ is written as $\H_1\oplus\H_2$.

\begin{exm}
Take $\mathcal{H}_i=\overline{K}^m_{n_i}$,~$1\leq{i}\leq{k}\leq{m}$,~$S=\{\mathcal{H}_i:i=1,2,\dots,k\}$ and $w_s=1$.
Then $\oplus_{\mathcal{H}_i\in{S}}^{1}{\mathcal{H}_i}=K^m_{n_1,n_2,\dots,n_k}$, which is the weak $m$-uniform $k$-partite complete hypergraph. 
Using the above corollary we get that for any $i=1,2,\dots,k$;~$-d^{S(m)}_{pp}$ is an
eigenvalue of $A_{K^m_{n_1,n_2\dots,n_k}}$ with the multiplicity atleast $(n_i-1)$ for $i=1,2,\dots,k$. The remaining eigenvalues of $A_{K^m_{n_1,n_2,\dots,n_k}}$ 
are the eigenvalues of the quotient matrix $B$, defined as 
$$
(B)_{pq}=
\begin{cases}
(n_p-1)d^{S(m)}_{pp} &\text{if $p=q$},\\
n_qd^{S(m)}_{pq} &\text{otherwise.}
\end{cases}
$$
\end{exm}

In the above two examples, if we take $k=m$, we have $0$ as an eigenvalue of $A_{K^{m}_{n_1,n_2,\dots,n_m}}$ with the multiplicity 
atleast $\sum\limits_{i=1}^mn_i-m$. Here, the quotient matrix formed by the equitable partition $\pi=\{V_1,V_2,\dots,V_m\}$ and which is given by
$$
B=\dfrac{1}{m-1}
\begin{bmatrix}
0&s_1&\cdots&s_1&s_1\\
s_2&0&\cdots&s_2&s_2\\
\vdots&\vdots&\ddots&\vdots&\vdots\\
s_m&s_m&\cdots&s_m&0
\end{bmatrix},
$$
where $s_i=\prod\limits_{j=1,j\neq{i}}^mn_j$.\\
Note that
$\alpha(\neq{0})\in{spec(K^m_{n_1,n_2,\dots,n_m})}$ if and only if $r^{m-1}\alpha\in{spec(K^m_{rn_1,rn_2,\dots,rn_m})}$ for $r\in{\mathbb{N}}$.

Using $(1)$ from \cite{Bapat2018}, we have the characteristic polynomial of $Q'_B:=(m-1)B^T$ is 
$\phi_{B}(x)=x^m-\sum\limits_{i=2}^m{(i-1)}\sigma_i(s_1,s_2,\dots,s_m)x^{m-i},$ 
where $\sigma_i(s_1,s_2,\dots,s_m)=\sum\limits_{1\leq{j_1}<j_2<\dots<j_i\leq{m}}{s_{j_1}s_{j_2}\dots{s_{j_i}}}.$\\
We may consider $K^m_{n_1,n_2,\dots,n_m}$ as the generalization of complete bipartite graphs. Now we have the following proposition 

\begin{proposition}
Characteristic polynomial of $K^m_{n_1,n_2,\dots,n_m}$ is $x^{n-m}\Bigg(x^m-\sum\limits_{i=2}^m\dfrac{i-1}{(m-1)^i}\sigma_i(s_1,s_2,\dots,s_m)x^{m-i}\Bigg)$
where $n=\sum\limits_{i=1}^m{n_i}$,~$s_i=\prod\limits_{j=1,j\neq{i}}^mn_j$ and
$\sigma_i(s_1,s_2,\dots,s_m)=\sum\limits_{1\leq{j_1}<j_2<\dots<j_i\leq{m}}{s_{j_1}s_{j_2}\dots{s_{j_i}}}$.
\end{proposition}

From the above proposition it is clear that the quotient matrix $B$ is non-singular. Hence the multiplicity of eigenvalue $0$ of 
$K^m_{n_1,n_2,\dots,n_m}$ is $n-m$.

\begin{remark}
Note that $spec(K^m_{1,1,\dots,1})=\{1,-\dfrac{1}{m-1},-\dfrac{1}{m-1},\dots,-\dfrac{1}{m-1}\}$. Thus the
non-zero eigenvalues of $K^m_{n,n,\dots,n}$ are $\{n^{m-1},\overbrace{-\dfrac{n^{m-1}}{m-1},-\dfrac{n^{m-1}}{m-1},\dots,-\dfrac{n^{m-1}}{m-1}}^{m-1}\}$.
\end{remark}

\textbf{Note:} Let $\H$ be the $m$-uniform $m$-partite hypergraph $K^m_{\underbrace{n_1,n_1,\dots,n_1}_{l_1},\underbrace{n_2,n_2,\dots,n_2}_{l_2}}$, where $l_1+l_2=m$. 
Then the quotient matrix $B$ for the equitable partition formed by the $m$-parts of $\H$ can be written as $B=\dfrac{1}{m-1}B'$, 
where $B'$ is given by
$$
\begin{bmatrix}
s_1(J_{l_1}-I_{l_1})& s_1J_{l_1\times{l_2}}\\
s_2J_{l_2\times{l_1}}& s_2(J_{l_2}-I_{l_2})
\end{bmatrix},
$$ where
$s_1=n_1^{l_1-1}n_2^{l_2}$,~$s_2=n_1^{l_1}n_2^{l_2-1}$. Now using the Lemma~\ref{lemma2}, we have the characteristic polynomial of $B'$ as follows
\begin{align*}
f_{B'}(x)&=det(B'-xI)\\
&=det\bigg(s_2J_{l_2}-(s_2-x)I\bigg)det\bigg(s_1J_{l_1}-(s_1+x)I-s_1s_2J_{l_1\times l_2}(s_2J_{l_2}-(s_2+x)I)^{-1}J_{l_2\times l_1}\bigg)\\
&=(-1)^{m-2}(x+s_1)^{l_1-1}(x+s_2)^{l_2-1}(x-a^{+})(x-a^{-}),
\end{align*}
where $a^{\pm}=\dfrac{1}{2}\Big{[}s_1(l_1-1)+s_2(l_2-1)\pm\sqrt{\{s_1(l_1-1)+s_2(l_2-1)\}^2+4s_1s_2(l_1+l_2-1)}\Big{]}$. Thus the eigenvalues of $\H$ are $\dfrac{-s_i}{m-1}$ with the multiplicity $l_i-1$ for $i=1,2$ and
 $\dfrac{a^{\pm}}{m-1}$ with the multiplicity $1$.

\subsection{Weighted Joining of a Set of Non-Uniform Hypergraphs}

Let $S=\{\mathcal{H}_i(V_i,E_i,W_i):i=1,2,\dots,k\}$ be a set of hypergraphs with $|V_i|=n_i$. Now for $T_s\subseteq{\{k,k+1,\dots,\sum\limits_{i=1}^kn_i\}}$,
we define the join $\H_{T_s}=\oplus^{T_s}_{\H_i\in{S}}\H_i$ as a hypergraph with the vertex set $V=\cup_{i=1}^kV_i$, the edge set $E=\cup_{i=0}^kE_i$, where 
$E_0=\{e\subseteq{V}:e\cap{V_i}\neq{\phi},\forall{i=1,2.\dots,k},|e|\in{T_s}\}$ and the weight function $W:E\rightarrow{\mathbb{R}_{\geq 0}}$ defined by
$$
W(e)=
\begin{cases}
W_i(e) &\text{if $e\in{E_i}$},\\
w_m &\text{if $e\in{E_0}$ and $|e|=m$,}
\end{cases}
$$
where $w_m$ are non-negative constants for each $m\in{T_s}.$
The adjacency matirx $A_{\mathcal{H}_{T_s}}$ of the hypergraph $\mathcal{H}_{T_s}$ is given by
$$
(A_{\mathcal{H}_{T_s}})_{ij}=
\begin{cases}
\left(A_{\mathcal{H}_p}\right)_{ij}+\sum\limits_{m\in{T_S}}w_md^{S(m)}_{pp} &\text{if $i, j\in{V_p}$,~$i\neq{j}$},\\
0 &\text{$i=j$,}\\
\sum\limits_{m\in{T_S}}w_md^{S(m)}_{pq} &\text{if $i\in{V_p}, j\in{V_q}, p\neq{q}$.}
\end{cases}
$$
When $\mathcal{H}_i$'s are regular, $\pi=\{V_1,V_2,\dots,V_k\}$ forms an
equitable partition for $\mathcal{H}_{T_s}$. If $\mathcal{H}_i$'s are $r_i$-regular then the quotient matrix $B$ is given by
$$
(B)_{pq}=
\begin{cases}
r_p+(n_p-1)\sum\limits_{m\in{T_S}}w_md^{S(m)}_{pp} &\text{if $p=q$,}\\
n_q\sum\limits_{m\in{T_S}}w_md^{S(m)}_{pq} &\text{otherwise.}
\end{cases}
$$

\subsubsection{Weighted Joining of Non-Uniform Hypergraphs on a Backbone Hypergraph}

Let $\H(V,E,W)$ be a hypergraph. We consider the hypergraph $\mathcal{H}_b(V_b,E_b,W_b)$, with the vertex set $V_b=\{1,2,\dots,n\}$
as a backbone of $\H$ if there exist a set $S=\{\H_i(V_i,E_i,W_i):i=1,2,\dots n\}$ of hypergraphs and for each $e\in{E_b}$ there exists a
subset
$T_e\subseteq{\{|\eta|: \eta\in{E}\}}$ such that $\H$ can be constructed by following steps
\begin{enumerate}
\item Replace vertex $i$ of $\H_b$ by the hypergraph $\H_i$.
\item For each edge $e=\{j_1,j_2,\dots,j_k\}\in{E_b}$, take $S_e=\{\mathcal{H}_{j_i}:i=1,2\dots,k\}$ and 
consider the operation $\oplus^{T_e}_{\H_i\in{S_e}}\H_i$
\item Consider $w_m$ as edge weight for every new edge of cardinality $m$. 
\end{enumerate}
Here  
$\mathcal{H}_i$'s are participants and we write $\H=(\H_b,S,T)$ where $T=\{T_e:e\in{E_b}\}$. 
The adjacency matrix $A_{\mathcal{H}}$ of $\mathcal{H}$ can be written by using the adjacency matrices of $\H_i$ as follows
$$
(A_{\mathcal{H}})_{ij}=
\begin{cases}
(A_{\mathcal{H}_p})_{ij}+\sum\limits_{e\in{E_B}, p\in{e}}\sum\limits_{m\in{T_e}}w_md^{S_e(m)}_{pp} &\text{if i, j$\in{V_p}$,~$i\neq{j}$,}\\
0 &\text{if $i=j$,}\\
\sum\limits_{e\in{E_B}, p, q\in{e}}\sum\limits_{m\in{T_e}}w_md^{S_e(m)}_{pq} &\text{if $i\in{V_p}$,~$j\in{V_q}$,~$p\neq{q}$.}
\end{cases}
$$
Now we have a theorem similar to the Theorem~\ref{th1}.  

\begin{theorem}\label{th3}
Let $\mathcal{H}_b(V_b,E_b,W_b)$ be a hypergraph with $|V_b|=n$. Let $S=\{\mathcal{H}_i(V_i,E_i):i=1,2,\dots,n\}$ be a collection of regular
hypergraphs and $\H=(H_b,S,T)$. Then for any non-perron eigenvalue $\lambda$ of $A_{\mathcal{H}_p}$ with the multiplicity
$l$,~$\lambda-\sum\limits_{e\in{E_b},p\in{e}}\sum\limits_{m\in{T_e}}w_md^{S_e{(m)}}_{pp}$ is an eigenvalue of $A_{\mathcal{H}}$ 
with the multiplicity atleast $l$.
\end{theorem}

The remaining $n$ eigenvalues of $A_{\mathcal{H}}$ are the eigenvalues of the 
quotient matrix $B$ formed the equitable partition $\pi=\{V_1,V_2,\dots,V_n\}$. If $\H_i$'s are $r_i$ regular, then  $B$ can be written as
$$
(B)_{pq}=
\begin{cases}
r_p+(n_p-1)\sum\limits_{e\in{E_B},p\in{e}}\sum\limits_{m\in{T_e}}w_md^{S_e(m)}_{pp} &\text{if $p=q$,}\\
n_q\sum\limits_{e\in{E_B},p,q\in{e}}\sum\limits_{m\in{T_e}}d^{S_e(m)}_{pq} &\text{otherwise.}
\end{cases}
$$

\section{Generalized Corona of Hypergraphs}
Let $\mathcal{H}(V,E)$ be an $m$-uniform hypergraph. A subhypergraph induced by $V^{'}\subset{V}$ is the hypergraph $\mathcal{H}[V^{'}]$ with the vertex
set $V^{'}$ and edge set $E^{'}=\{e: e\in{E}, e\subset{V^{'}}\}$. Now for $V^{'}\subset{V}$ and a hypergraph $\H''(V'',E'')$, we denote $V^{'}\oplus{\mathcal{H}^{''}}$ as
the hypergraph with the vertex set $V^{'}\cup{V^{''}}$ and edge set $E(\mathcal{H}[V^{'}]\oplus{\mathcal{H}^{''}})\cup{E(\mathcal{H})}$.

\subsection{Vertex Corona}
Let $\mathcal{H}_0(V_0,E_0)$ be an $m$-uniform hypergraph and $\pi=\{V^{(1)}_0,V^{(2)}_0,\dots,V^{(k)}_0\}$ be a partition of $V_0=\{1,2,\dots,n(=pk)\}$
with $V^{(i)}_0=\{(i-1)p+1,(i-1)p+2,\dots,ip\}$ for $i=1,2,\dots,k$. Also let $\{\mathcal{H}_i(V_i,E_i):1\leq{i}\leq{k}\}$ be a set of $m$-uniform
hypergraphs with $|V_i|=n_i$. For each $i=1,2,\dots,k$, we take $p$ copies, $\{\mathcal{H}^{(j)}_i(V^{(j)}_i,E^{(j)}_i):j=1,2\dots,p\}$,
of $\mathcal{H}_i(V_i,E_i)$. Then we consider $V^{(i)}_0\oplus{\mathcal{H}^{j}_i(V^{(j)}_i,E^{(j)}_i)}$ for all $i=1,2,\dots,k$,~$j=1,2,\dots,p$.
This gives us an $m$-uniform hypergraph $\H_{\pi}(V,E)$ which is called \textit{generalized corona} of hypergraphs and we write 
$\mathcal{H}_{\pi}=\mathcal{H}_0\circ^k_p{\mathcal{H}_i}$.\\

Here, we consider the case when $n_i=n_1$ and $\H_i$'s are $r_1$-regular. Now to find the characteristic polynomial we have the following theorem. We denote
$a=\dfrac{p}{m-1}\Big[\binom{n+n_1-2}{m-2}-\binom{n}{m-2}\Big]$,~$b=\binom{n+n_1-2}{m-1}$,
~$c=\dfrac{1}{m-1}\Big[\binom{n+n_1-2}{m-2}-\binom{n_1-2}{m-2}\Big]$,~$D_i=A_{\mathcal{H}_i}+c(J_{n_1}-I_{n_1})$,~
$D=diag(I_p\otimes{D_1},I_p\otimes{D_2},\dots,I_p\otimes{D_k})$ and $S=I_k\otimes{J_{p\times{pn_1}}}$.
%%%%%%%%%%%%%%%%%%%%
The kronecker product $A\otimes B$ between two matrices $A=(a_{ij})$ and $B=(b_{pq})$ is defined as the partition matrix $(a_{ij}B)$. For matrices $A,B,C$ and $D$ we have $AB\otimes CD=(A\otimes C)(B\otimes D)$, when multiplication makes sense. Now we use the following lemma to prove the next theorem.

\begin{lemma}[Schur complement, \cite{Bapat2010}]\label{lemma2}
Let $A$ be an $n\times{n}$ matrix partitioned as 
	$$
	A=
	\begin{bmatrix}
	A_{11}&A_{12}\\
	A_{21}&A_{22}
	\end{bmatrix}
	$$
	where $A_{11}$ and $A_{22}$ are square matrices. If $A_{11}$ and $A_{22}$ are invertible, then
	\begin{align}
	det(A)&=det(A_{11})det(A_{22}-A_{12}A^{-1}_{11}a_{21})\\
	&=det(A_{22})det(A_{11}-A_{21}A^{-1}_{11}A_{12}).
	\end{align}
\end{lemma}

\begin{theorem}\label{th4}
Characteristic polynomial of $A_{\H_{\pi}}$ can be expressed as 
\begin{align}\label{Eq3}
f_{\H_{\pi}}(x)=\bigg(\prod\limits_{i=1}^kdet(D_i-xI_{n_1})\bigg)^pdet\bigg(A_{\H_0}+I_k\otimes{((a-\dfrac{b^2pn_1}{r_1+(n_1-1)c-x})J_p-(a+x)I_p)}\bigg).
\end{align}
\end{theorem}

\begin{proof}

The adjacency matrix of $\H_{\pi}$ is given by
$$
A_{\H_{\pi}}=
\begin{bmatrix}
A_{\H_0}+aI_k\otimes{(J_{p}-I_p)}&bS\\
bS^T&D\\
\end{bmatrix}.
$$
Using Lemma~\ref{lemma2} we have
\begin{align}
f_{\H_{\pi}}(x)&=det(A_{\H_{\pi}}-xI) \notag \\
&=det(D-xI)det\{A_{\H_0}+aI_k\otimes{(J_p-I_p)}-b^2S(D-xI)^{-1}S^T\}.\label{Eq4}
\end{align}
Again,
\begin{align*}
(D-xI)S^T&=diag(I_p\otimes(D_1-xI_{n_1}),I_p\otimes(D_2-xI_{n_1}),\dots,I_p\otimes(D_k-xI_{n_1}))(bI_k\otimes{J_{pn_1\times{p}}})\\
&=b\{r_1+c(n_1-1)-x\}(I_k\otimes{J_{pn_1\times{p}}})\\
&=b\{r_1+c(n_1-1)-x\}S^T.
\end{align*}
Therefore,
\begin{align}\label{Eq5}
S^T=\dfrac{1}{r_1+(n_1-1)c-x}(D-xI)S^T.
\end{align}
Also
\begin{align}\label{Eq6}
SS^T=pn_1I_k\otimes{J_p}.
\end{align}
Using the Equations (\ref{Eq5}) and (\ref{Eq6}) in the Equation (\ref{Eq4}) we get the Equation (\ref{Eq3}).

\end{proof}

\begin{remark}
For two regular hypergraphs $\H_1,\H_2$, we can find eigenvalues of $\H:=\H_1\oplus\H_2$ by using Theorem \ref{th1}. Again we can find the characteristic polynomial of $\H$. If $spec(A_{\H_i})=\{\lambda^{(i)}_1,\lambda^{(i)}_2,\dots,\lambda^{(i)}_{n_i}(=r_i)\}$
for $i=1,2$ then the characteristic polynomial of $A_{\H}$ is given by
\begin{align*}
f_{\H}(x)=f_B(x)\prod\limits_{i=1}^{n_1-1}(\lambda^{(1)}_i-d^{11}_{(m)}-x)\prod\limits_{i=1}^{n_2-1}(\lambda^{(2)}_i-d^{22}_{(m)}-x),
\end{align*}
where $f_B(x)=det(B-xI)$.
\end{remark}

\begin{corollary}\label{cor3}
Let p=1 and $spec(A_{\H_0})=\{\mu_i:i=1,2\dots,n\}$,~$spec(A_{\H_i})=\{\lambda^{(i)}_1,\lambda^{(i)}_2,\dots,\lambda^{(i)}_{n_1}(=r_1)\}$.
Then the adjacency eigenvalues of $\H_{\pi}=\H_0\circ^n_1{\H_i}$ are $\lambda^{(i)}_j-c$ with the multiplicity one for $i=1,2,\dots,n$~$j=1,2,\dots,n_1-1$ and $\alpha^{\pm}_i$ with the multiplicity one for $i=1,2,\dots,n$
where $\alpha^{\pm}_i=\dfrac{1}{2}\Big[r_1+(n_1-1)c+\mu_i\pm\sqrt{\{r_1+(n_1-1)c-\mu_i\}^2+4b^2n_1}\Big]$.

\end{corollary}

\begin{proof}
For $p=1$,~$k=n$, we have,
\begin{align*}
&A_{\H_0}+I_k\otimes{((a-\dfrac{b^2pn_1}{r_1+(n_1-1)c-x})J_p-(a+x)I_p)}=:X(say) \\
&=A_{\H_0}+\dfrac{x^2-(r_1-(n_1-1)c)x-b^2n_1}{r_1+(n_1-1)c-x}I_n.
\end{align*}
Now, $spec(X)=\{\mu_i+\dfrac{x^2-(r_1-(n_1-1)c)x-b^2n_1}{r_1+(n_1-1)c-x}:i=1,2,\dots,n\}$.\\
Thus, 
\begin{align}\label{Eq8}
det(X)&=\prod\limits_{i=1}^n{\{\mu_i+\dfrac{x^2-(r_1-(n_1-1)c)x-b^2n_1}{r_1+(n_1-1)c-x}\}} \notag \\
&=\dfrac{\prod\limits_{i=1}^n{(x-\alpha^{+}_i)(x-\alpha^{-}_i)}}{\{r_1+(n_1-1)c-x\}^n}.
\end{align}
Again for each $i=1,2,\dots,n$
\begin{align*}
spec(D_i-xI_{n_1})=\{\lambda^{(i)}_j-c-x:j=1,2\dots,n_1-1\}\cup\{r_1+(n_1-1)c-x\}.
\end{align*}
We have
\begin{align}\label{Eq9}
\prod\limits_{i=1}^ndet(D_i-xI_{n_1})=\{r_1+(n_1-1)c-x\}^n\prod\limits_{i=1}^n\prod\limits_{j=1}^{n_1-1}(\lambda^{(i)}_j-c-x).
\end{align}
Using the Equations $(\ref{Eq8})$ and (\ref{Eq9}) in the Equation (\ref{Eq3}) we have our desired result.
\end{proof}

Note that for two non-isomorphic hypergraphs $\mathcal{H}_0$ and $\mathcal{G}_0$ we get the hypergraphs $\mathcal{H}=\mathcal{H}_0\circ^k\mathcal{H}_i$ and
$\mathcal{G}=\mathcal{G}_0\circ^k{\mathcal{H}_i}$ which are also non-isomorphic. So by using the Corollary~\ref{cor3}, for any two non-isomorphic co-spectral hypergraphs
$\mathcal{H}_0$ and $\mathcal{G}_0$, we can find infinitely many pair of non-isomorphic co-spectral hypergraphs. To start with, first we can construct a pair of non-isomorphic co-spectral hypergraphs using the procedure similar to the Godsil-McKay switching \cite{Godsil1982, Zhou} as follows: \\

Let $\mathcal{H}(V,E)$ be an m-uniform hypergraph. Also let $\rho=\{V_1,V_2,\dots,V_k,D\}$,~$(where |D|=m-1,|V_i|=2t$,~$t\in{\mathbb N}$, for all $i=1,2,\dots,k)$
 be a partition of $V$ satisties the following conditions
 \begin{enumerate}
\item Let $p,q\in\{1,2,\dots,k\}$. Then for any $i,j\in{V_p}$
$$\sum_{l\in{V_q}}A_{il}=\sum_{l\in{V_q}}A_{jl}=B'_{pq}.$$
Also for any $i,j\in{V_q}$,
$$\sum_{l\in{V_p}}A_{il}=\sum_{l\in{V_p}}A_{jl}=B'_{qp}$$ and $B'_{pq}=B'_{qp}$.
\item
Let $D=\{u_1,u_2,\dots,u_{m-1}\}$. For each $e\in{E}$,~$|e\cap{D}|=0$ or $m-1$. Also $|N(u_1,u_2,\dots,u_{m-1})\cap{V_p}|=0,t$
or $2t$; where $N(u_1,u_2,\dots,u_{m-1})=\{v\in{V}:\{u_1,u_2,\dots,u_{m-1},v\}\in{E}\}$.
\end{enumerate}
Now we consider  $p\in\{1,2,\dots,k\}$ for which $|N(u_1,u_2,\dots,u_{m-1})\cap{V_p}|=t$. Suppose that $V_p=\{v^{(p)}_1,v^{(p)}_2,\dots,v^{(p)}_{2t}\}$ and
$N(u_1,u_2,\dots,u_{m-1})\cap{V_p}=\{v^{(p)}_1,v^{(p)}_2,\dots,v^{(p)}_t\}$.
We remove the edges $\{u_1,u_2, \dots, $ $u_{m-1},v^{(p)}_i\}$ for $i=1,2,\dots,t$ and consider 
$\{u_1,u_2,\dots,u_{m-1},v^{(p)}_j\}$ for $j=t+1,t+2,\dots,2t$ as new edges. In this way, from $\mathcal{H}$ we get a new 
$m$-uniform hypergraph $\mathcal{H}_{\rho}$.

\begin{proposition}\label{prop3}
Let $\H$ and $\H_{\rho}$ be the $m$-uniform hypergraphs as described above. Then $\H$ and $\H_{\rho}$ are cospectral.
\end{proposition}

\begin{proof}
After labeling the vertices as in the partition $\rho$ we can write the adjacency matrix of $\H$ as
$$
A_{\mathcal{H}}=\dfrac{1}{m-1}
\begin{bmatrix}
M_{11}&M_{12}\cdots&M_{1k}&D_{11}\\
M^T_{12}&M_{22}\cdots&M_{2k}&D_{21}\\
\vdots&\vdots \ddots&\vdots&\vdots\\
M^T_{2k}&M^T_{2k}\cdots&M_{kk}&D_{k1}\\
D^T_{11}&D^T_{21}\cdots&D^T_{k1}&D_0
\end{bmatrix},
$$
where for $i\in\{1,2,\dots,k\}$,~$D_{i1}=J_{2t\times{(m-1)}}$ or $D_{i1}=[X^{(i)}_1,X^{(i)}_2,\dots,X^{(i)}_{m-1}]$, where $X^{(i)}_j{(j=1,2,\dots,m-1)}$
are the same column vector consisting $t$-zeros and $t$-ones and $D_0=\alpha(J-I)$ for some $\alpha\in{\mathbb{N}}$.\\
We take $X=diag(X_{2t},X_{2t},\dots,X_{2t},I_{m-1})$ where $X_{2t}=\dfrac{2}{2t}J_{2t}-I_{2t}$. Note that for $D_{i1}=J_{2t\times{(m-1)}}$, 
we have $D^T_{i1}X_{2t}=D^T_{i1}$ and when $D_{i1}=[X^{(i)}_1,X^{(i)}_2,\dots,X^{(i)}_{m-1}]$, we have $D^T_{i1}X_{2t}=J_{(m-1)\times{2t}}-D^T_{i1}$.
Also, each $M_{ij}$ has constant row sum and column sum and thus $X_{2t}M_{ij}X_{2t}=M_{ij}$. Hence
$XA_{\mathcal{H}}X=A_{\mathcal{H}_{\rho}}$, where $X^2=I_n$ and which implies $A_{\mathcal{H}}$,~$A_{\mathcal{H}_{\rho}}$ are similar.
\end{proof}

Take a regular hypergraph $\mathcal{H}_1(V_1,E_1)$ with $V_1=\{v_1,v_2,\dots,v_{2t}\}$ and $D=\{u_1,u_2,\dots,u_{m-1}\}$. 
For any subset $V_2\subset{V_1}$ with $|V_2|=t$ we suppose $E_2=\{e_i: e_i=\{u_1,u_2,\dots,u_{m-1},v_i\}, v_i\in{V_2}\}$, $E_3=\{e'_i: e'_i=\{u_1,u_2,\dots,u_{m-1},v_j\}, v_j\in{V_1\backslash{V_2}}\}$. Now the  Proposition~\ref{prop3}, implies that
$\mathcal{H}(V_1\cup{D},E_1\cup{E_2})$ and $\mathcal{H}_{\rho}(V_1\cup{D},E_1\cup{E_3})$ are cospectral.

\begin{exm}
Let $V_0=\{1,2,\dots,8\}$, $E^{(1)}_0=\{\{1,2,3\},\{3,4,5\},\{5,6,1\},\{2,4,6\},\{7,8,3\},\{7,8,4\},$ $\{7,8,5\}\}$ and 
$E^{(2)}_0=\{\{1,2,3\},\{3,4,5\},\{5,6,1\},\{2,4,6\},\{7,8,1\},\{7,8,2\},\{7,8,6\}\}$. 
Then $\mathcal{H}_0(V_0,E^{(1)}_0)$ and $\mathcal{G}_0(V_0,E^{(1)}_0)$ are non-isomorphic cospectral 3-uniform hypergraphs.
Here we take $\mathcal{H}_1(V_1,E_1)$ as $V_1=\{1,2,3,4,5,6\}$,~$E_1=\{\{1,2,3\},\{3,4,5\},\{5,6,1\},\{2,4,6\}\}$,~$D=\{7,8\}$,~$V_2=\{3,4,5\}$.
\end{exm}

Now we have another corollary of the Theorem~\ref{th4}.

\begin{corollary}\label{cor4}
Let $k=1$ and $\mathcal{H}_0$ be $r_0$-regular. Let $spec(A_{\mathcal{H}_0})=\{\mu_1,\mu_2,\dots,\mu_n(=r_0)\}$,~
$spec(A_{\mathcal{H}_1})=\{\lambda_1,\lambda_2,\dots,\lambda_{n_1}(=r_1)\}$ and $\mathcal{H}_{\pi}=\mathcal{H}_0\circ^1{\mathcal{H}_1}$. Then the adjacency eigenvalues of $\H_{\pi}$ are given by
$r_1+(n_1-1)c$ with the multiplicity $n-1$,~$\lambda_i$ with the multiplicity $n$ for $i=1,2,\dots,n_1-1$,~$\mu_j$ with the multiplicity one for $j=1,2,\dots,n-1$ and $\alpha^{\pm}$ with the multiplicity one 
where $\alpha^{\pm}=\dfrac{1}{2}\Big[r_1+(n_1-1)c+r_0+(n-1)a\pm\sqrt{\{r_1-r_0+(n_1-1)c-(n-1)a\}^2+4b^2nn_1}\Big]$.

\end{corollary}

\begin{proof}
For $p=n,k=1$, we have
\begin{align*}
&A_{\H_0}+I_k\otimes{\Bigg((a-\dfrac{b^2nn_1}{r_1+(n_1-1)c-x})J_p-(a+x)I_p\Bigg)}=:Y \\
&=A_{\H_0}+\Bigg(a-\dfrac{b^2nn_1}{r_1+(n_1-1)c-x}\Bigg)J_n-(a+x)I_n.
\end{align*}
Now $spec(Y)=\{\mu_i-a-x:i=1,2,\dots,n-1\}\cup{\{\mu\}}$, where 
\begin{align*}
\mu&=r_0+n\bigg(a-\dfrac{b^2nn_1}{r_1+(n_1-1)c-x}\bigg)-(a+x)\\
&=r_0+(n-1)a-\dfrac{b^2nn_1}{r_1+(n_1-1)c-x}-x\\
&=\dfrac{x^2-\{r_1+(n_1-1)c+r_0+(n-1)a\}x+(r_1+(n_1-1)c)(r_0+(n-1)a)-b^2nn_1}{r_1+(n_1-1)c-x}\\
&=\dfrac{(x-\alpha^{+})(x-{\alpha}^{-})}{r_1+(n_1-1)c-x}.
\end{align*}
Therefore, 
\begin{align}
det(Y)=\dfrac{(x-\alpha^{+})(x-{\alpha}^{-})}{r_1+(n_1-1)c-x}\prod\limits_{i=1}^{n-1}(\mu_i-a-x).
\end{align}
We also have 
\begin{align}
det(D_1-xI_{n_1})=\{r_1+(n_1-1)c-x\}\prod\limits_{i=1}^{n_1-1}(\lambda_i-c-x).
\end{align}
From the Equation $(\ref{Eq3})$ we get 
\begin{align*}
f_{\mathcal{H}_{\pi}}(x)=\{det(D_1-xI_{n_1})\}^ndet(Y).
\end{align*}
This completes the proof.

\end{proof}

\subsection{Edge Corona}
Let $\H_0(V_0,E_0)$ be an $m$-uniform hypergraph with the edge set $E_0=\{e_1,e_2,\dots,e_k\}$ and $|V_0|=n_0$. Also let $\H_i(V_i,E_i)$~$(i=1,2,\dots,k)$ be 
$m$-uniform hypergraphs. For each $i=1,2,\dots,k$ we consider $e_i\oplus\H_i$. The new hypergraph is known as the \textit{edge corona} of 
hypergraphs and we write it by $\H=\H_0\square^k\H_i$. 
When $|V_i|=n_1$ for all $i=1,\dots,k$, we write $D_i=A_{\mathcal{H}_i}+c(J_{n_1}-I_{n_1})$, take $D=diag(D_1,D_2,\dots,D_k)$,~$R=$ vertex-edge incidence 
matrix, for $\mathcal{H}_0$,~$1_{n_1}=[1,1,\dots,1]$ row vector of length $n_1$,~$a=\binom{m+n_1-2}{m-2}-1$,
~$b=\dfrac{1}{m-1}\binom{n+n_1-2}{m-2}$,~$c=\binom{m+n_1-2}{m-2}-\binom{n_1-2}{m-2}$.
Now using Lemma~\ref{lemma2} we have the following theorem

\begin{theorem}\label{th5}
Let $\mathcal{H}_0(V_0,E_0)$ be an $m$-uniform hypergraph with $|V_0|=n$,~$|E_0|=k$. Let $\{\mathcal{H}_i(V_i,E_i):1\leq{i}\leq{k},|V_i|=n_1\}$
be a set of $m$-uniform hypergraphs. Then the the characteristic polynomial of $A_{\mathcal{H}}$ for the edge corona $\mathcal{H}=\mathcal{H}_0\square^k{\mathcal{H}_i}$
is as follows
\begin{align}\label{Eq11}
f_{\mathcal{H}}(x)=det(D-xI_{kn_1})det(\{(a+1)A_{\mathcal{H}_0}-xI_n-b^2(R\otimes{1_{n_1}})(D-xI_{kn_1})^{-1}(R^T\otimes{1^T_{n_1}})\}).
\end{align}
\end{theorem}

\begin{corollary}\label{cor5}
Let $\mathcal{H}_i$'s be $r_1$-regular hypergraphs and $spec(A_{\mathcal{H}_i})=\{\lambda^{(1)}_i,\lambda^{(2)}_i,\dots,\lambda^{(n_1)}_i(=r_1)\}$. 
Then the characteristic polynomial of $\H$ can be given by
\begin{align}\label{Eq12}
f_{\mathcal{H}}(x)&=\{r_1+(n_1-1)c-x\}^{k-n}det\bigg(\beta_1(x)A_{\mathcal{H}_0}-b^2n_1D_d+\beta_2(x)I_n\bigg)\prod\limits_{i=1}^k\prod\limits_{j=1}^{n_1-1}(\lambda^{(j)}_i-c-x),
\end{align} where $D_d=diag(d_1,d_2,\dots,d_n)$ where $d_i$ denote the degree of vertices of $A_{\H_0}$,~$\beta_1(x)=(a+1)\{r_1+(n_1-1)c-x\}-(m-1)b^2n_1$,~
 $\beta_2(x)=x\{x-r_1-(n_1-1)c\}$.
\end{corollary}

\begin{proof}
Note that $(D-xI_{kn_1})(R^t\otimes{1^t_{n_1}})=(r_1+(n_1-1)c-x)(R^t\otimes{1^t_{n_1}})$ and $RR^T=D_d+(m-1)A_{\mathcal{H}_0}$. 
Therefore $$(R\otimes{1_{n_1}})(D-xI_{kn_1})^{-1}(R^T\otimes{1^T_{n_1}})=\dfrac{n_1}{r_1+(n_1-1)c-x}[D_d+(m-1)A_{\mathcal{H}_0}].$$\\ 
Then we have
\begin{align}\label{Eq13}
det\bigg((a+1)A_{\mathcal{H}_0}-xI_n-b^2(R\otimes{1_{n_1}})(D-xI_{kn_1})^{-1}(R^T\otimes{1^T_{n_1}})\bigg)
&=\dfrac{det\bigg(\beta_1(x)A_{\mathcal{H}_0}-b^2n_1D_d+\beta_2(x)I_n\bigg)}{\{r_1+(n_1-1)c-x\}^n}.
\end{align}
Again 
\begin{align}\label{Eq14}
det(D-xI_{kn_1})&=\prod\limits_{i=1}^k{det(D_i-xI_{n_1})}  \notag\\
&=\Big\{\prod\limits_{i=1}^k\prod\limits_{j=1}^{n_1-1}(\lambda^{(j)}_i-c-x)\Big\}\{r_1+(n_1-1)c-x\}^k.
\end{align}
Using the Equations~(\ref{Eq13}) and (\ref{Eq14}) in the Equation~(\ref{Eq11}) we have the Equation (\ref{Eq12}).
\end{proof}

\begin{corollary}\label{cor6}
Let $\mathcal{H}_i$'s be the hypergraphs mentioned in the above corollary and $\mathcal{H}_0$ be $r$-regular with $spec(A_{\mathcal{H}_0})=\{\mu_1,\mu_2,\dots,\mu_n(=r)\}$.
Then the adjacency eigenvalues of $\H$ are
$r_1+(n_1-1)c$ with the multiplicity $k-n$,~$\lambda^{(j)}_i$ with the multiplicity one, for all $i=1,2,\dots,k$,~$j=1,2,\dots,n_1-1$ and $\beta^{\pm}_{j}$ with the multiplicity one for $j=1,2,\dots,n$, where ${\beta_j}^{\pm}=\dfrac{1}{2}\Big[r_1+(n_1-1)c+(a+1)\mu_j\pm\sqrt{\{r_1+(n_1-1)c-(a+1)\mu_j\}^2+4b^2n_1\{(m-1)\mu_j+r\}}\Big]$. 
 
\end{corollary}

\begin{proof}
Here we have, $D_d=rI_n$. Thus
\begin{align*}
det\bigg(\beta_1(x)A_{\H_0}-b^2n_1D_d+\beta_2(x)I_n\bigg)&=det(\{\beta_1(x)A_{\H_0}+(\beta_2(x)-b^2rn_1)I_n\}) \notag \\
&=\prod\limits_{j=1}^n\{\beta_1(x)\mu_j+\beta_2(x)-b^2rn_1\}.
\end{align*}
Again $\beta_1(x)\mu_j+\beta_2(x)-b^2rn_1=(x-{\beta_j}^{+})(x-{\beta_j}^{-})$ and hence the result follows.
\end{proof}

\subsubsection{Loose Cycles and Loose Paths}
Recently loose cycles and loose paths attract  interest of many researchers (\cite{Qi2014}, \cite{Luke2019}). In \cite{Luke2019}
the authers consider oriented hypergraphs and they aim to find the spectra of loose cycles and loose paths, where the concept of hypergraph duality is used. They consider $-(m-1)A_{\H}$ as the adjacency matrix of an $m$-uniform hypergraph $\H$. The adjacency eigenvalues of $C^{(2s)}_{L(s,n)}$ and $C^{(s+1)}_{L(s,s+1)}$ have been derived. They also poses the question: ``What are the adjacency eigenvalues of $C^{(m)}_{L(s;n)}$ other than the above two cases ?''.
Thus here we are ineterested to find the eigenvalues of our adjacency matrix of $C^{(m)}_{L(s;n)}$ when $m\geq{2s}$.

\begin{theorem}\label{th6}
The adjacency eigenvalues of an $s$-loose cycle $C^{(m)}_{L(s;n)}$, are
\begin{enumerate}
\item $\dfrac{-2}{2s-1}$ with the multiplicity $n(s-1)$ and 
$\dfrac{2}{2s-1}(s-1+s\cos{\dfrac{2\pi i}{n}})$ with the multiplicity one, for $i=1,2,\dots,n$, when $m=2s$ and
\item \label{case2} 
$\dfrac{-1}{m-1}$ with the multiplicity atleast $n(m-2s-1)$,
$\dfrac{-2}{m-1}$ with the multiplicity atleast $n(s-1)$ and
$\gamma^{+}_i$,~$\gamma^{-}_i$ with the multiplicity atleast one, where, $$\gamma_i^{\pm}=\dfrac{1}{2}\Bigg[m-3+2s\cos{\dfrac{2\pi{i}}{n}} \pm\sqrt{(m-3+2s\cos{\dfrac{2\pi{i}}{n}})^2+8(m-s-1+s\cos{\dfrac{2\pi{i}}{n}})}\Bigg],$$ for $i=1,2,\dots,n$, when $m\geq{2s+1}$.

\end{enumerate}
\end{theorem}
\begin{proof}
\hspace{-0.05cm}
\begin{enumerate}
\item \textbf{Case $m=2s$:} Let $\H=C^{(2s)}_{L(s;n)}$.
In the Corollary~\ref{cor2} take $G_b=C_n$, the cycle graph over $n$ vertices, and $\H_i=K_{s}$, the complete graph with $s$-vertices with each edge weight two. Then the resultant hypergraph is a graph, $G$(say). Hence $A_{\H}=\dfrac{1}{2s-1}A_G$. Thus $\dfrac{-2}{2s-1}$ is an eigenvalue of $A_{\H}$ with the multiplicity atleast $n(s-1)$. The quotient matrix is $B=\dfrac{1}{2s-1}\{sA_{C_n}+2(s-1)I_n\}$. The remaining eigenvalues of $A_{\H}$ are $\dfrac{2}{2s-1}(s-1+s\cos{\dfrac{2\pi i}{n}})$ for $i=1,2,\dots,n$.
\item \textbf{Case $m\geq 2s+1$:}
Let $G_b=C_n\square^{n}K_1$ and $\H=C^{(m)}_{L(s;n)}$. We take the vertices of $G_b$ as $V(C_n)=\{1,2,\dots,n\}$ and $V(G_b)\backslash{V(C_n)}=\{n+1,n+2,\dots,2n\}$.
For $i=1,2,\dots,n$, we take $G_i=K_s$, the complete graph with $s$ vertices with edge weight $1$, and for $i=n+1,\dots,2n$, take $G_i=K_{m-2s}$
with edge weight 2. Considering $G_B$ as backbone graph with each edge weight one and $G_i$'s as participants, we get a graph $G$ (say). Then $A_{\H}=\dfrac{1}{m-1}A_{G}$. Now using the Corollary~\ref{cor2} we get the eigenvalues of $A_{G}$ which are -1 with the  multiplicity atleast
$n(m-2s-1)$ and $-2$ with multiplicity $n(s-1)$.

Next using Equation~(\ref{Qmatrix}) we have the remaining $2n$ eigenvalues are the eigenvalues of the quotient matrix
$B$ given by
$$
(B)_{pq}=\dfrac{1}{m-1}
\begin{cases}
r_q &\text{if $p=q$},\\
n_q &\text{if $p\sim{q}$ in $G_B$,}\\
0 &\text{otherwise,}
\end{cases}
$$
where $r_q=2s-2, n_q=s$ for $q=1,2,\dots,n$ and $r_q=m-2s-1, n_q=m-2s$ for $q=n+1,\dots,2n$.
We write 
$$
B=\dfrac{1}{m-1}
\begin{bmatrix}
(2s-2)I_n+sA_{C_n} &t(I_n+Y)\\
s(I_n+Y^t) &(t-1)I_n
\end{bmatrix},
$$
where $t=m-2s$,~$A_{C_n}$ is the adjacency matrix of an $n$-cycle $C_n$ and $Y$ is the $n\times{n}$ circulant matrix with the first row $[0,0,\dots,1]$. 
We suppose $B=\dfrac{1}{m-1}B'$. Then using Lemma~\ref{lemma2} and the fact $(I_n+Y)(I_n+Y^t)=2I_n+A_{C_n}$ we have the cahracteristic polynomial 
of $B'$, as follows 
\begin{align}\label{eq_22}
f_{B'}(x)&=det(B'-xI_n) \notag \\
&=det(\{(t-1-x)I_n\})det(\{sA_{C_n}+(2s-2-x)I_n-\dfrac{st}{t-1-x}(I_n+Y)(I_n+Y^t)\}) \notag \\
&=det(\{x^2-(2s+t-3)x-(2s+2t-2)\}I_n-(s+sx)A_{C_n}).
\end{align} 
The eigenvalues of $A_{C_n}$ are $\mu_i=2\cos{\dfrac{2\pi{i}}{n}}$,~$i=1,2\dots,n$. Thus from the  Equation~(\ref{eq_22}) we have
\begin{align}
f_{B'}(x)&=\prod\limits_{i=1}^{n}\{x^2-(2s+t-3+s\mu_i)x-(2s+2t-2+s\mu_i)\} \notag \\
&=\prod\limits_{i=1}^{n}(x-\gamma_i^{+})(x-\gamma_i^{-}).
\end{align}
\end{enumerate}
\end{proof}
In the second case of the Theorem~\ref{th6}, if we take $s=1$, we get the adjacency eigenvalues of loose cycle $C^m_{L(1;n)}$ and which are $\dfrac{-1}{m-1}$ with the  multiplicity $n(m-3)$ and
$$\dfrac{1}{2}\Bigg[m-3+2\cos{\dfrac{2\pi{i}}{n}} \pm\sqrt{(m-3+2\cos{\dfrac{2\pi{i}}{n}})^2+8(m-2+\cos{\dfrac{2\pi{i}}{n}})}\Bigg]$$ 
with the multiplicity one for $i=1,2,\dots,n$.
We can also prove this result by using edge corona. For that, let $\H=C^m_{L(1;n)}$ and $G=C_n\square K_{m-2}$. Then $A_{\H}=\dfrac{1}{m-1}A_G$. We know that  
$spec(C_n)=\{\mu_i=2\cos{\dfrac{2\pi{i}}{n}}:i=1,2\dots,n\}$. We also have
$spec(K_{m-2})=\{m-3,-1,-1,\dots,-1\}$. Now using the Corollary~\ref{cor5} for graphs we have the eigenvalues of $A_{G}$ and hence the eigenvalues of $A_{\H}$ which are the same as given above.

First part of the Theorem~\ref{th6} provides an alternative prove of the part $(1)$ of the Theorem 3.6 in \cite{Luke2019}. Now we pose the following question.
\begin{Question}
	What are the adjacency eigenvalues of   $C^m_{L(s;n)}$ for $m\leq 2s-1$?
\end{Question}

\begin{lemma}\label{Eq16}
	For a square matrix $A$ we have
	\begin{align}
	det(A+\sum\limits_{i=1,n}u_{ii}E_{ii})=det(A)+\sum\limits_{i=1,n}u_{ii}det(A(i|i))+u_{11}u_{nn}det(A(1,n|1,n)),
	\end{align}
	where $A(i|j)$ is the matrix obtained from $A$ by deleting the $i$-th row and $j$-th column, respectively, and $E_{i,j}$ is the matrix with $1$ in $(i,j)$-th position and zero elsewhere.
\end{lemma}

\begin{theorem}\label{th7}
 The adjacency eigenvalues of an $s$-loose path $P^m_{L(s,n)}$ are
 \begin{enumerate}
 \item
$\dfrac{-1}{m-1}$ with the multiplicity atleast $n(m-1)-2s(n-1)$,~$\dfrac{-2}{m-1}$ with the  multiplicity atleast $(n-1)(s-1)$ and 
$\dfrac{\alpha_i}{m-1}$ with the multiplicity one, for $i=1,2,\dots,2n-1$, where $\alpha_i$'s are the zeros of the polynomial 
 $$
\dfrac{(m-s-1-x)^2f_1(x)+2s^2(1+x)(m-s-1-x)f_2(x)+s^4(1+x)^2f_3(x)}{m-2s-1-x},
$$ 
where 
$$
f_j(x)=\prod\limits_{i=1}^{n-j}\{x^2-(m-3+2s\cos{\dfrac{\pi i}{n-j+1}})x-2(m-s-1-x+s\cos{\dfrac{\pi i}{n-j+1}})\}.
$$ for $j=1,2,3$, when $m\geq{2s+1}$ and
\item $\dfrac{-1}{m-1}$ with the multiplicity $2(s-1)$,~$\dfrac{-2}{m-1}$ with multiplicity $(n-1)(s-1)$ and $\dfrac{\beta_i}{2s-1}$ with the multiplicity one, where $\beta_i$ are the zeros of the polynomial $(x-s+1)^2t_1(x)+2s^2(x-s+1)t_2(x)+s^4t_3(x),$ where $t_j(x)=\prod\limits_{i=1}^{n-j}(2s-2-x+2s\cos{\dfrac{\pi i}{n-j+1}})$, for $j=1,2,3$ when $m=2s.$
\end{enumerate}
\end{theorem}
\begin{proof}
\hspace{-.05cm}
\begin{enumerate}
\item \textbf{Case $m\geq 2s+1$:}
 Let $G_0=P_{n-1}\square K_1$ and $\H=P^{(m)}_{L(s;n)}$. We label the veretices of $G_0$ as the vertices of $P_{n-1}$ are labelled as $\{1,2,\dots,n-1\}$ and $V(G_{0}\backslash V(P_{n-1}))=\{n,n+1,\dots,2n-3\}$. Now we join the vertex $2n-2$ with the vertex $1$ and the vertex $2n-1$ with the vertex $n-1$ by an edge respectively. 
 Let $G_b(V_b,E_b)$ be the new resulting graph. We also take
 $$
 G_i=
 \begin{cases}
 K_s&\text{with each edge weight 2, for $i=1,2,\dots,n-1$},\\
 K_{m-2s}&\text{with each edge weight 1, for $i=n,n+1,\dots,2n-3$},\\
 K_{m-s}&\text{with each edge weight 1, for $i=2n-2,2n-1$.}
 \end{cases}
 $$
 Now considering $G_b$ as the backbone and $G_i$'s as the participants we get a graph $G$(say). 
 Here $A_{\H}=\dfrac{1}{m-1}A_G$. Using Corollary~\ref{cor2} we have $-1$ and $-2$ are the eigenvalues of $G$ with the multiplicities atleast $n(m-1)-2s(n-1)$ and atleast $(n-1)(s-1)$ respectively. The remaining eigenvalues of $A_G$ are the eigenvalues of the quotient matrix $B$ which is given by 
 $$
 B=
 \begin{bmatrix}
 sA_{n-1}+2(s-1)I_{n-1}&(m-2s)Y&(m-s)Z\\
 sY^t&(m-2s-1)I_{n-2}&0\\
 sZ^t&0&(m-s-1)I_2
 \end{bmatrix},
 $$ where $A_l$ is the adjacency matrix of the path graph $P_l$,~
 $
 Y=
 \begin{bmatrix}
 1&0&0&\cdots&0&0\\
 1&1&0&\cdots&0&0\\
 0&1&1&\cdots&0&0\\
 \vdots&\vdots&\vdots&\ddots&\vdots&\vdots\\
 0&0&0&\cdots&1&1\\
 0&0&0&\cdots&0&1
 \end{bmatrix}_{n-1\times n-2}
 $ and
 $
 Z=
 \begin{bmatrix}
 1&0&\cdots &0\\
 0&0&\cdots &1
 \end{bmatrix}^t
 $.
 Since $ZZ^t=E_{11}+E_{n-1n-1}$ and $YY^t=2I_{n-1}-E_{11}-E_{n-1n-1}+A_{n-1}$, using Lemma~ \ref{lemma2} we get the characteristic polynomial of $B$ as
 \begin{align*}
 f_B(x)&=det(B-xI)\\
 &=(m-s-1-x)^2(m-2s-1-x)^{n-2}det(\{sA_{n-1}+(2s-2-x)I_{n-1}-\\
 &\ \ \ \ \ \ 
 \begin{bmatrix}
 (m-2s)Y&(m-s)Z
 \end{bmatrix}diag(\dfrac{1}{m-2s-1-x}I_{n-2},\dfrac{1}{m-s-1-x}I_2) 
 \begin{bmatrix}
 sY^t\\
 sZ^t
 \end{bmatrix}\})\\
 &=(m-s-1-x)^2(m-2s-1-x)^{n-2}det(\{
 sA_{n-1}+(2s-2-x)I_{n-1}-\\
 &\ \ \ \ \ \
 \dfrac{s(m-s)}{m-s-1-x}ZZ^t-\dfrac{(m-2s)s}{m-2s-1-x}YY^t)\}\\
&=(m-s-1-x)^2(m-2s-1-x)^{n-2}det(\{sA_{n-1}+(2s-2-x)I_{n-1}\\
 &\ \ \ \ \ \ 
 -\dfrac{s(m-s)}{m-s-1-x}(E_{11}+E_{n-1n-1})-\dfrac{(m-2s)s}{m-2s-1-x}(2I_{n-1}-E_{11}-E_{n-1n-1}+A_{n-1}))\} \notag \\
 &=(m-s-1-x)^2(m-2s-1-x)^{n-2}det(\{t_1A_{n-1}+t_2I_{n-1}+t_3E_{11}+t_3E_{n-1n-1}\})\\
 &=(m-s-1-x)^2(m-2s-1-x)^{n-2}h(x),
 \end{align*}
 where $t_1=\dfrac{-s(1+x)}{m-2s-1-x}$,~$t_2=2s-2-x-\dfrac{2(m-2s)s}{m-2s-1-x}$,~$t_3=\dfrac{s^2(1+x)}{(m-2s-1-x)(m-s-1-x)}$ and $h(x)=det(\{t_1A_{n-1}+t_2I_{n-1}+t_3E_{11}+t_3E_{n-1n-1}\})$. 
 Now using the Lemma \ref{Eq16} we have
 \begin{align} 
 h(x)&=det(\{t_1A_{n-1}+t_2I_{n-1}+t_3E_{11}+t_3E_{n-1n-1}\}) \notag \\
 &=det(t_1A_{n-1}+t_2I_{n-1}+t_3E_{11})+t_3det(\{t_1A_{n-1}+t_2I_{n-1}+t_3E_{11}\}(n-1|n-1)) \notag \\
 &=det(t_1A_{n-1}+t_2I_{n-1})+t_3det(\{t_1A_{n-1}+t_2I_{n-1}\}(1|1))+ \notag \\
 & \notag \ \ \ \ \
 t_3det(\{t_1A_{n-1}+t_2I_{n-1}\}(n-1|n-1))+t_3^2det(\{t_1A_{n-1}+t_2I_{n-1}\}(1,n-1|1,n-1)) \notag \\
 &=det(t_1A_{n-1}+t_2I_{n-1})+2t_3det(\{t_1A_{n-2}+t_2I_{n-2}\}+t_3^2det(\{t_1A_{n-3}+t_2I_{n-3}\}.
 \end{align} 
 Again we know that the adjacency eigenvalues of a path $P_l$ of length $l$ are $\{2\cos \dfrac{\pi i}{l+1}: i=1,2,\dots,l\}$.
 Thus we get
 \begin{align}
 h(x)&=\prod\limits_{i=1}^{n-1}(t_2+2t_1\cos{\dfrac{\pi i}{n}})+2t_3\prod\limits_{i=1}^{n-2}(t_2+2t_1\cos{\dfrac{\pi i}{n-1}})+t_3^2\prod\limits_{i=1}^{n-3}(t_2+2t_1\cos{\dfrac{\pi i}{n-2}}) \notag \\
 &=h_1(x)+2t_3h_2(x)+t_3^2h_3(x),
 \end{align} 
 where $h_j(x)=\dfrac{f_j(x)}{(m-2s-1-x)^{n-j}}$ for $j=1,2,3$. \\
 
 Now, $h(x)=\dfrac{(m-s-1-x)^2f_1(x)+2s^2(1+x)(m-s-1-x)f_2(x)+s^4(1+x)^2f_3(x)}{(m-2s-1-x)^{n-1}(m-s-1-x)^2}$. \\
 Therefore $$f_B(x)=\dfrac{(m-s-1-x)^2f_1(x)+2s^2(1+x)(m-s-1-x)f_2(x)+s^4(1+x)^2f_3(x)}{m-2s-1-x}.$$
 \item \textbf{Case $m=2s$:} Let $\H=P^{(2s)}_{L(s;n)}$ and $G_b=P_{n+1}$. We take the vertices of $G_b$ as $V(G_b)=\{1,2,\dots,n+1\}$ with the end vertices $1$ and $n+1$. Here we take 
 $$
 G_i=
 \begin{cases}
 K_s&\text{with each edge weight 2, for $i=2,3,\dots n$},\\
 K_s&\text{with each edge weight 1, for $i=1,n+1$}.
 \end{cases}
 $$ 
 Now, considering $G_b$ as the backbone and $G_i$'s as the participants we get a graph $G$(say). Then $A_{\H}=\dfrac{1}{2s-1}A_G$. The quotient matrix $B$ is given by 
 $$
 B=
 \begin{bmatrix}
 sA_{n-1}+2(s-1)I_{n-1}&sZ\\
 sZ^t&(s-1)I_2
 \end{bmatrix},
 $$ where $Z$ is the matrix defined in first part of this theorem. Using Lemma~\ref{lemma2} we have the characteristic polynomial of $B$ as
 \begin{align*}
 f_B(x)&=det(B-xI)\\
 &=(s-1-x)^2det(\{sA_{n-1}+(2s-2-x)I_{n-1}-\dfrac{s^2}{s-1-x}(E_{11}+E_{n-1n-1})\}).
 \end{align*}
 Now, the result folows by using the similar technique of computation shown in part one.
 \end{enumerate}
 \end{proof}
 We consider a special case of the Theorem~\ref{th7} by taking $s=1$.  
Let $\H=P^{(m)}_{L_{(1,n)}}$ be the $m$-uniform loose path with $n$-edges. Let $G=P_{n+1}\square^{n}K_{m-2}$. 
Then $A_{\H}=\dfrac{1}{m-1}A_{G}$. Using the Corollary \ref{cor5} we have the characteristic polynomial of $A_G$ as
\begin{align}
f_G(x)&=det(A_G-xI) \notag \\
&=\dfrac{(-1-x)^{n(m-3)}}{m-3-x}det(\{(m-3-x)A_{n+1}-(m-2)D_d+x(x-m+3)I_{n+1}\})\label{Eq15}\\
&=\dfrac{(-1-x)^{n(m-3)}}{m-3-x}g(x), &\text{(say)} \notag
\end{align}
where $D_d=diag(1,2,2,\dots,2,1)$ and $A_{n+1}$ is the adjacency matrix of the path graph $P_{n+1}$ over $n+1$ vertices. Again, $D_d=2I_{n+1}-E_{11}-E_{n+1n+1}$. So we have 
\begin{align*}
g(x)&=det(\{(m-3-x)A_{P_{n+1}}-(m-2)D_d+x(x-m+3)I_{n+1}\})\\
&=det(\{(m-3-x)A_{P_{n+1}}-(m-2)(2I_{n+1}-E_{11}-E_{n+1n+1})+x(x-m+3)I_{n+1}\})\\
&=det(\{uA_{P_{n+1}}+vI_{n+1}+(m-2)E_{11}+(m-2)E_{n+1n+1}\}),\\
&\ \ \ \
\text{where $u:=m-3-x, v:=x(x-m+3)-2(m-2)$}\\
&=det(\{uA_{P_{n+1}}+vI_{n+1}+(m-2)E_{11}\})+(m-2)det(\{uA_{P_{n+1}}+vI_{n+1}+(m-2)E_{11}\}(n+1|n+1))\\
&=det(\{uA_{P_{n+1}}+vI_{n+1}\})+(m-2)det(\{uA_{P_{n+1}}+vI_{n+1}\}(1|1))\\  &\ \ \ \ + (m-2)det(\{uA_{P_n}+vI_n+(m-2)E_{11}\})\\
&=det(\{uA_{P_{n+1}}+vI_{n+1}\})+2(m-2)det(\{uA_{P_n}+vI_n\})+(m-2)^2det(\{uA_{P_{n-1}}+vI_{n-1}\})\\
&=\prod\limits_{i=1}^{n+1}(2u\cos\dfrac{\pi i}{n+2}+v)+2(m-2)\prod\limits_{i=1}^n(2u\cos\dfrac{\pi i}{n+1}+v)+(m-2)^2\prod\limits_{i=1}^{n-1}(2u\cos\dfrac{\pi i}{n}+v)\\
&=g_1(x)+2(m-2)g_2(x)+(m-2)^2g_3(x).
\end{align*}
\begin{align*}
g_j(x)&=\prod\limits_{i=1}^{n-j+2}(2u\cos\dfrac{\pi i}{n-j+3}+v)\\
&=\prod\limits_{i=1}^{n-j+2}\{(x-m+3)(x-\cos{\dfrac{\pi i}{n-j+3}})-2(m-2)\},
\end{align*} where for $j=1,2,3$.

Thus we find the eigenvalues of the graph $G$ as the zeros of the polynomial 

\begin{align}
f_G(x)=\dfrac{(-1-x)^{n(m-3)}}{m-3-x}\{g_1(x)+2(m-2)g_2(x)+(m-2)^2g_3(x)\}.
\end{align}

From the Theorem~\ref{th7} we have $(x+1)^2$ is a factor of $g(x)$.

\section*{Acknowledgements}
AS is sincerely thankful to Shibananda Biswas for fruitful discussions.  AB is thankful to the Science and Engineering Research Board (SERB), Government of India for the financial support through Mathematical Research Impact Centric Support (MATRICS) grant (File Number: MTR/2017/000988).

\end{document}